\documentclass[submission]{FPSAC2020}
\usepackage{amsmath, amssymb, amsthm, graphicx, mathtools, bm, bbm, relsize, subcaption}


\usepackage[style=numeric, backend=bibtex8, citestyle=numeric, firstinits=true, maxnames=99, sorting=nyt]{biblatex}
\DeclareFieldFormat[article, inbook, incollection, inproceedings, patent, thesis, unpublished]{title}{#1} 
\renewbibmacro{in:}{}

\theoremstyle{plain}
\newtheorem{thm}{Theorem}[section]
\newtheorem{lem}[thm]{Lemma}

\newtheorem{conj}[thm]{Conjecture}

\theoremstyle{definition}

\newtheorem{ex}[thm]{Example}

\usepackage[dvipsnames]{xcolor}

\DeclarePairedDelimiter\abs{\lvert}{\rvert} 
\makeatletter
\let\oldabs\abs
\def\abs{\@ifstar{\oldabs}{\oldabs*}}

\newcommand{\firstmention}{\emph}
\newcommand{\N}{\mathbb{N}}

\newcommand{\C}{\mathbb{C}}
\newcommand{\sym}{S}
\newcommand{\freevecspace}{\mathbb{C}_{\mathbf{x}}^{n-1}S}
\newcommand{\syt}{\textnormal{SYT}}
\newcommand{\staircase}{\delta}
\newcommand{\young}{\mathcal{Y}}

\newcommand{\sort}{\textnormal{SN}}
\newcommand{\eg}{\textnormal{EG}}
\newcommand{\diag}{\textnormal{cor}}
\newcommand{\diagsorted}{\overline{\diag}}

\newcommand{\rev}{\textnormal{rev}}
\newcommand{\id}{\textnormal{id}}
\newcommand{\fin}{\textnormal{last}}
\newcommand{\finsorted}{\overline{\fin}}
\newcommand{\tab}{\textnormal{Tab}}
\newcommand{\interlacingtab}{\textnormal{IntTab}}
\newcommand{\Rnonneg}{\mathbb{R}_{\geq 0}}
\newcommand{\Znonneg}{\mathbb{Z}_{\geq 0}}

\renewcommand{\P}{\mathbb{P}}

\newcommand{\1}{\mathbbm{1}} 
\DeclareMathOperator{\e}{\mathrm{e}}
\newcommand{\RSK}{\mathsf{RSK}}
\newcommand{\Burge}{\mathsf{Bur}}

\newcommand{\Fro}{\mathcal{B}}

\DeclareMathSymbol{\widehatsym}{\mathord}{largesymbols}{"62}

\renewcommand{\hat}{\widehat}

\title[Sorting networks, staircase Young tableaux and last passage percolation]{Sorting networks, staircase Young tableaux and last passage percolation}

\author[E.~Bisi, F.~D.~Cunden, S.~Gibbons, \and D.~Romik]{Elia Bisi\thanks{\href{mailto:elia.bisi@ucd.ie}{elia.bisi@ucd.ie}. Supported by ERC Advanced Grant \emph{IntRanSt} - 669306.}\addressmark{1},
Fabio Deelan Cunden\thanks{\href{mailto:fabio.cunden@ucd.ie}{fabio.cunden@ucd.ie}. Supported by ERC Advanced Grant \emph{IntRanSt} - 669306 and GNFM-INdAM.}\addressmark{1},
Shane Gibbons\thanks{\href{mailto:shane.gibbons@ucdconnect.ie}{shane.gibbons@ucdconnect.ie}. Supported by UCD Undergraduate Summer Research programme.}\addressmark{1},
\and Dan Romik\thanks{\href{mailto:romik@math.ucdavis.edu}{romik@math.ucdavis.edu}. Supported by the National Science Foundation grant No.\ DMS-1800725.}\addressmark{2}}

\address{\addressmark{1}School of Mathematics and Statistics, University College Dublin\\ \addressmark{2}Department of Mathematics, University of California, Davis}

\received{25 October 2019}

\abstract{
We present new combinatorial and probabilistic identities relating three random processes: the oriented swap process on $n$ particles, the corner growth process, and the last passage percolation model. We prove one of the probabilistic identities, relating a random vector of last passage percolation times to its dual, using the duality between the Robinson--Schensted--Knuth and Burge correspondences. A second probabilistic identity, relating those two vectors to a vector of ``last swap times'' in the oriented swap process, is conjectural. We give a computer-assisted proof of this identity for $n\le 6$ after first reformulating it as a purely combinatorial identity, and discuss its relation to the Edelman--Greene correspondence.
}


\begin{document}
\maketitle

\section{Introduction}

\subsection{Random growth models and algebraic combinatorics}

Randomly growing Young diagrams, and the related models known as \firstmention{Last Passage Percolation} (LPP) and the \firstmention{Totally Asymmetric Simple Exclusion Process} (TASEP), are intensively studied stochastic processes.
Their analysis has revealed many rich connections to the combinatorics of Young tableaux, longest increasing subsequences, the Robinson--Schensted--Knuth (RSK) algorithm, and related topics---see for example~\cite[Chs.~4-5]{romik15}.

\firstmention{Random sorting networks} are another family of random processes.
Two main models, the \firstmention{uniform random sorting networks} and the \firstmention{Oriented Swap Process} (OSP), have been analyzed~\cite{angelDauvergneEtAl17, angelHolroydRomikVirag07, angelHolroydRomik09, dauvergne18, dauvergneVirag18} and are known to have connections to the TASEP, last passage percolation, and also to staircase shape Young tableaux via the \firstmention{Edelman--Greene bijection}~\cite{edelmanGreene87}.

In this extended abstract we discuss a new and surprising meeting point between the aforementioned subjects.
In an attempt to address an open problem from~\cite{angelHolroydRomik09} concerning the absorbing time of the OSP, we discovered elegant distributional identities relating the oriented swap process to last passage percolation, and last passage percolation to itself.
We will prove one of the two main identities; the other one is a conjecture that we have been able to verify for small values of a parameter~$n$.
As shown below, the analysis relies in a natural way on well-known notions of algebraic combinatorics, namely the RSK, Burge and Edelman--Greene correspondences.

Our conjectured identity apparently requires new combinatorics to be explained, and has far-reaching consequences for the asymptotic behavior of the OSP as the number of particles grows to infinity (see the remarks following Theorem~\ref{thm:main-thm2} below).
This extended abstract focuses on combinatorial ideas and requires only a minimal background in the relevant probabilistic concepts.
The full version will include additional probabilistic and asymptotic results, and the proof details that are omitted or only sketched here.

\subsection{Two probabilistic identities}

The two main identities presented in this paper take the form 
$$
\bm{U}_n \overset{D}{=} \bm{V}_n \overset{D}{=}
\bm{W}_n \, ,
$$
where $ \overset{D}{=}$ denotes equality in distribution, and $\bm{U}_n$, $\bm{V}_n$, $\bm{W}_n$ are $(n-1)$-dimensional random vectors associated with the following three random processes.

\medskip \noindent \textbf{The oriented swap process.}
This process~\cite{angelHolroydRomik09} describes randomly sorting a list of~$n$~particles labelled $1,\ldots,n$.
At time $t=0$, particle labelled~$j$ is in position~$j$ on the finite integer lattice $[1,n]=\{1,\ldots,n\}$.
All pairs of adjacent positions $k,k+1$ of the lattice are assigned independent Poisson clocks.
The system then evolves according to the random dynamics whereby each pair of particles with labels $i,j$ occupying respective positions $k$, $k+1$ attempt to swap when the corresponding Poisson clock rings; the swap succeeds only if $i<j$, i.e.,\ if the swap increases the number of inversions in the sequence of particle labels.
We then define the vector $\bm{U}_n = (U_n(1), \ldots, U_n(n-1))$ of \firstmention{last swap times} by
\begin{align*}
U_n(k) &:= \text{the last time $t$ at which a swap occurs between positions $k$ and $k+1$}.
\end{align*}

As explained in~\cite{angelHolroydRomik09}, the last swap times are related to the \firstmention{particle finishing times}: it is easy to see that $\max\{U_n(n-k), U_n(n-k+1)\}$ is the finishing time of particle $k$ (with the convention that $U_n(0)=U_n(n)=0$).
The random variable $\max_{1\le k\le n-1} U_n(k)$ is the absorbing time of the process, whose limiting distribution was the subject of an open problem in~\cite{angelHolroydRomik09} (see also~\cite[Ex.~5.22(e), p.~331]{romik15}).

\medskip \noindent \textbf{Randomly growing a staircase shape Young diagram.}
In this process, a variant of the \firstmention{corner growth process}, starting from the empty Young diagram, boxes are successively added at random times, one box at each step, to form a larger diagram until the staircase shape $\staircase_n=(n-1,n-2,\ldots,1)$ is formed.
We identify each box of a Young diagram $\lambda$ with the position $(i,j)\in \N^2$, where $i$ and $j$ are the row and column index respectively.
All boxes are assigned independent Poisson clocks.
Each box $(i,j)\in \staircase_n$, according to its Poisson clock, attempts to add itself to the current diagram $\lambda$, succeeding if and only if $\lambda \cup \{(i,j)\}$ is still a Young diagram.
The vector $\bm{V}_n = (V_n(1), \ldots, V_n(n-1))$ records when boxes along the $(n-1)$th anti-diagonal are added:
$$ 
V_n(k) := \text{the time at which the box at position $(n-k,k)$ is added.}
$$

\smallskip \noindent \textbf{The last passage percolation model.}
This process describes the maximal time spent travelling from one vertex to another of the two-dimensional integer lattice along a directed path in a random environment. 
Let $(X_{i,j})_{i,j\ge 1}$ be an array of independent and identically distributed (i.i.d.)~exponential random variables of rate $1$, referred to as \firstmention{weights}.
For $(a,b), (c,d)\in \N^2$, define a \firstmention{directed lattice path} from $(a,b)$ to $(c,d)$ to be any sequence $\big( (i_k,j_k) \big)_{k=0}^m$ of minimal length $|c-a|+|d-b|$ such that $(i_0,j_0)=(a,b)$, $(i_m,j_m)=(c,d)$, and $\abs{i_{k+1}-i_k} + \abs{j_{k+1}-j_k} =1$ for all $0\leq k<m$.
We then define the \firstmention{Last Passage Percolation} (LPP) time from $(a,b)$ to $(c,d)$ as
\begin{equation}
\label{eq:lpp}
L(a,b;c,d) := \max_{\pi \colon (a,b)\to (c,d)} \sum_{(i,j)\in \pi} X_{i,j} \, ,
\end{equation}
where the maximum is over all directed lattice paths $\pi$ from $(a,b)$ to $(c,d)$.

The LPP model has a precise connection (see \cite[Ch.~4]{romik15}) with the corner growth process, whereby each random variable $L(1,1;i,j)$ is the time when box $(i,j)$ is added to the randomly growing Young diagram.
We can thus equivalently define $\bm{V}_n$ in terms of the LPP times between the vertices $(1,1)$ and $(i,j)$ of the  lattice $[1,i]\times[1,j]$, where $i+j=n$:
\begin{equation}
\label{eq:def-vn}
\bm{V}_n = (L(1,1; n-1,1), L(1,1; n-2,2), \ldots, L(1,1; 1,n-1) ) \, .
\end{equation}
On the other hand, we can now consider the ``dual'' LPP times between the other two vertices of the same rectangles, and define $ \bm{W}_n = (W_n(1), \ldots, W_n(n-1))$ to be
\begin{equation}
\label{eq:def-wn}
\bm{W}_n := (L(n-1,1; 1,1), L(n-2,1; 1,2), \ldots, L(1,1;1,n-1) ) \, .
\end{equation}

With these definitions, we have the following results.

\begin{thm}
\label{thm:main-thm1}
$\bm{V}_n \overset{D}{=} \bm{W}_n$ for all $n\ge2$.
\end{thm}

\begin{conj}
\label{main-conj}
$\bm{U}_n \overset{D}{=} \bm{V}_n$  for all $n\ge2$.
\end{conj}

Conjecture~\ref{main-conj} is consistent with (and implies) the marginal identities proved in~\cite{angelHolroydRomik09}.
\begin{thm}[Angel, Holroyd, Romik, 2009]
\label{thm:main-thm3}
$\bm{U}_n(k) \overset{D}{=} \bm{V}_n(k)$  for  all $1\le k\le n-1$, $n \ge 2$.
\end{thm}
We were able to prove  the conjectured identity for small values of $n$.
\begin{thm}
\label{thm:main-thm2}
$\bm{U}_n \overset{D}{=} \bm{V}_n$  for  $2\le n\le 6$.
\end{thm}

Theorem~\ref{thm:main-thm1} is proved in Section~\ref{lpp}.
As we will see, the distributional identity $\bm{V}_n \overset{D}{=} \bm{W}_n$ arises as a special case of a more general family of identities (Theorem~\ref{thm:UpDownLPP}) involving LPP times between pairs of opposite vertices in rectangles $[1,i]\times [1,j]$, where each $(i,j)$ belongs to the so-called border strip of a Young diagram.
This result is, in turn, a consequence of the duality between the RSK and Burge correspondences, and holds also in the discrete setting where the weights $X_{i,j}$ follow a geometric distribution.

The above results have an important consequence in the asymptotic analysis of the OSP.
As we will explain in the full version of this extended abstract, Conjecture~\ref{main-conj} implies that the total absorbing time of the OSP on $n$ particles is distributed as a so-called point-to-line LPP model studied in~\cite{bisiZygouras19b}, thus exhibiting fluctuations of order $n^{1/3}$ and GOE Tracy-Widom limiting distribution as $n\to\infty$. This solves, modulo Conjecture~\ref{main-conj}, an open problem posed in~\cite{angelHolroydRomik09}.

\subsection{A combinatorial reformulation of Conjecture~\ref{main-conj}}

The conjectural equality in distribution between $\bm{U}_n$ and $\bm{V}_n$ remains mysterious, but we made some progress towards understanding its meaning by reformulating it as an algebraic-combinatorial identity that is of independent interest.
\begin{conj} 
\label{main-conj-reformulated}
For $n\ge 2$ we have the identity of vector-valued generating functions
\begin{equation}
\label{eq:comb-iden}
\sum_{t\in \syt(\staircase_n)} f_t(x_1,\ldots,x_{n-1}) \sigma_t
= \sum_{s\in \sort_n} 
g_s(x_1,\ldots,x_{n-1}) \pi_s \, .
\end{equation}
\end{conj}
Precise definitions and examples will be given in Section~\ref{sec:comb-iden}, where we will prove the equivalence between Conjectures~\ref{main-conj} and~\ref{main-conj-reformulated}.
For the moment, we only remark that the sums on the left-hand and right-hand sides of~\eqref{eq:comb-iden} range over the sets of staircase shape standard Young tableaux $t$ and sorting networks $s$ of order~$n$, respectively; $f_t$ and $g_s$ are certain rational functions, and $\sigma_t$, $\pi_s\in \sym_{n-1}$ are permutations associated with $t$ and $s$. 

The identity \eqref{eq:comb-iden} reduces the proof of $\bm{U}_n \overset{D}{=} \bm{V}_n$ for fixed~$n$ to a concrete finite computation.
This enabled us to provide a computer-assisted verification of Conjecture~\ref{main-conj} for $4\le n\le 6$ (the cases $n=2,3$ can be checked by hand) and thus prove Theorem~\ref{thm:main-thm2}.

\section{Equidistribution of LPP times and dual LPP times along border strips}

\label{lpp}

The goal of this section is to prove Theorem~\ref{thm:main-thm1}.
We first fix some terminology.
We say that $(i,j)$ is a \firstmention{border box} of a Young diagram $\lambda$ if $(i+1,j+1)\notin \lambda$, or equivalently if $(i,j)$ is the last box of its diagonal.
We refer to the set of border boxes of $\lambda$ as the \emph{border strip} of $\lambda$.
We say that $(i,j)\in \lambda$ is a \emph{corner} of $\lambda$ if $\lambda\setminus\{(i,j)\}$ is a Young diagram.
Note that every corner is a border box.
We refer to any array $x=\{x_{i,j}\colon (i,j)\in\lambda\}$ of non-negative real numbers as a \emph{tableau} of shape $\lambda$.
We call $x$ an \emph{interlacing tableau} if its diagonals interlace, in the sense that
$x_{i-1,j} \leq x_{i,j}$ if $i>1$ and 
$x_{i,j-1} \leq x_{i,j}$ if $j>1$
for all $(i,j)\in\lambda$, or equivalently if its entries are weakly increasing along rows and columns.

Let now $X$ be a \emph{random} tableau of shape $\lambda$ with non-negative random entries $X_{i,j}$.
We can then define the associated LPP time $L(a,b;c,d)$ between two boxes $(a,b),(c,d)\in\lambda$ as in~\eqref{eq:lpp}.
We will mainly be interested in the special $\lambda$-shaped tableaux
$L = (L_{i,j})_{(i,j) \in \lambda}$ and $L^* = (L^*_{i,j})_{(i,j) \in \lambda}$, which we respectively call the \firstmention{LPP tableau} and the \firstmention{dual LPP tableau}, defined by
$$
L_{i,j} = L(1,1; i,j) \, , 
\qquad
L^*_{i,j} = L(i,1; 1,j) \, , \qquad \text{for } (i,j)\in \lambda \, .
$$
It is easy to see from the definitions that $L$ and $L^*$ are both (random) interlacing tableaux.

Now, if the weights are i.i.d., it is evident that, for each $(i,j)\in \lambda$, the distributions of $L_{i,j}$ and $L^*_{i,j}$ coincide.
Remarkably, if the common distribution of the weights is geometric or exponential, then a far stronger distributional identity holds:

\begin{thm}
\label{thm:UpDownLPP}
Let $X$ be a Young tableau of shape $\lambda$ with i.i.d.\ geometric or i.i.d.\ exponential weights.
Then the border strip entries (and in particular the corner entries) of the corresponding LPP and dual LPP tableaux $L$ and $L^*$ have the same joint distribution.
\end{thm}

Theorem~\ref{thm:main-thm1} immediately follows from Theorem~\ref{thm:UpDownLPP} applied to tableaux of staircase shape $(n-1,n-2,\dots,1)$, since in this case the coordinates of $\bm{V}_n$ and $\bm{W}_n$ are precisely the corner entries of $L$ and $L^*$, respectively.

We will sketch the proof of Theorem~\ref{thm:UpDownLPP} by using an extension of two celebrated combinatorial maps, the Robinson--Schensted--Knuth and Burge correspondences~\cite{fulton97}.
We denote by $\tab_{S}(\lambda)$ the set of tableaux of shape $\lambda$ with entries in $S\subseteq \Rnonneg$, and by $\interlacingtab_{S}(\lambda)$ the subset of interlacing tableaux.
Let $\Pi^{(k)}_{m,n}$ be the set of all unions of $k$ disjoint non-intersecting directed lattice paths $\pi_1, \dots, \pi_k$ with $\pi_i$ starting at $(1,i)$ and ending at $(m,n-k+i)$.
Similarly, let $\Pi^{*(k)}_{m,n}$ be the set of all unions of $k$ disjoint non-intersecting directed lattice paths $\pi_1,\dots,\pi_k$ with $\pi_i$ starting at $(m,i)$ and ending at $(1,n-k+i)$.
\begin{thm}[\cite{bisiOConnellZygouras20, greene74, krattenthaler06}]
\label{thm:RSK}
Let $\lambda$ be a Young diagram with border strip $\Fro$ and let $S$ be one of the sets $\mathbb{Z}_{\ge 0}$ or $\mathbb{R}_{\ge 0}$.
There exist two bijections 
\begin{align*}
\RSK \colon \tab_{S}(\lambda)&\to \interlacingtab_{S}(\lambda)\, ,
& x = \{x_{i,j} \colon (i,j)\in \lambda\}
&\xmapsto{\RSK}
r = \{r_{i,j} \colon (i,j)\in \lambda\} \, , \\
\Burge \colon\tab_{S}(\lambda)&\to \interlacingtab_{S}(\lambda)\, ,
& x = \{x_{i,j} \colon (i,j)\in \lambda\}
&\xmapsto{\,\Burge\,}
b = \{b_{i,j} \colon (i,j)\in \lambda\} \, ,
\end{align*}
called the Robinson--Schensted--Knuth and Burge correspondences, that are characterized (and in fact defined) by the following relations:
for any $(m,n)\in \Fro$ and $1\leq k\leq \min(m,n)$,
\begin{equation}
\label{eq:RSK_Burge_Greene}
\sum_{i=1}^k r_{m-i+1,n-i+1}
= \max_{\pi \in\Pi^{(k)}_{m,n}} \sum_{(i,j)\in \pi} x_{i,j} \, , \qquad
\sum_{i=1}^k b_{m-i+1,n-i+1}
= \max_{\pi \in\Pi^{*(k)}_{m,n}} \sum_{(i,j)\in \pi} x_{i,j} \, .
\end{equation}
\end{thm}

The classical RSK and the Burge correspondences are known as bijections between non-negative integer matrices and  pairs of semistandard Young tableaux of the same shape.
Here we presented them, according to the construction in~\cite[\S~2]{bisiOConnellZygouras20}, in a somewhat untraditional way as bijections between tableaux and interlacing tableaux (with non-negative entries).
More details will be provided in the journal version of this extended abstract.
We only mention that the outputs $r_{i,j}$ and $b_{i,j}$ defined via~\eqref{eq:RSK_Burge_Greene} encode the shapes of the pair of tableaux obtained by applying the classical RSK and Burge correspondences to certain rectangular subarrays of $x$.
In particular, for a \emph{rectangular} Young diagram $\lambda$, the above result is essentially Greene's Theorem~\cite{greene74}.

In the extremal case $k=\min(m,n)$, both maxima in~\eqref{eq:RSK_Burge_Greene} equal $\sum_{i=1}^m \sum_{j=1}^n x_{i,j}$.
Moreover, the ``global'' sum $\sum_{(i,j)\in \lambda} x_{i,j}$ of the tableau $x$ can be expressed as a linear combination with integer coefficients of the ``rectangular'' sums $\sum_{i=1}^m \sum_{j=1}^n x_{i,j}$, $(m,n)\in\Fro$.
We thus deduce a crucial fact: for any shape $\lambda$ there exist integers $\{\omega_{i,j}\colon (i,j)\in \lambda\}$ such that
\begin{equation}
\label{eq:sumTableau}
\sum_{(i,j)\in \lambda} \omega_{i,j} r_{i,j}
= \sum_{(i,j)\in \lambda} x_{i,j}
= \sum_{(i,j)\in \lambda} \omega_{i,j} b_{i,j}
\end{equation}
for all $x\in\tab_{S}(\lambda)$, where $r:= \RSK(x)$ and $b:=\Burge(x)$.

\begin{lem}
\label{lem:RSK=Burge}
If $X$ is a random tableau of shape $\lambda$ with i.i.d.\ geometric or i.i.d.\ exponential entries, then
\begin{equation}
\RSK(X)\overset{D}{=} \Burge(X) \, .
\label{eq:RSK=Burge}
\end{equation}
\end{lem}

\begin{proof}
Assume first that $X$ has i.i.d.\ geometric entries with parameter $p\in (0,1)$, i.e., that $\P(X_{i,j} = m) = p (1-p)^m$ for all $m\ge 0$.
Choose $S= \Znonneg$ in Theorem~\ref{thm:RSK}.
Fix a tableau $t\in\interlacingtab_{\Znonneg}(\lambda)$ and let $y:=\RSK^{-1}(t)$ and $z:=\Burge^{-1}(t)$.
It then follows from~\eqref{eq:sumTableau} that
\[
\begin{split}
\P(\RSK(X)=t)
= \P(X=y)
&= p^{\abs{\lambda}} (1-p)^{\sum_{(i,j)\in\lambda} y_{i,j}}
= p^{\abs{\lambda}} (1-p)^{\sum_{(i,j)\in\lambda} \omega_{i,j} t_{i,j}} \\
&= p^{\abs{\lambda}} (1-p)^{\sum_{(i,j)\in\lambda} z_{i,j}}
= \P(X=z)
= \P(\Burge(X)=t) \, .
\end{split}
\]
This proves that $\RSK(X)$ and $\Burge(X)$ are equal in distribution, as claimed, for geometric weights.
By scaling the parameter as $p=\e^{-\epsilon \alpha}$ and taking the limit $\epsilon \downarrow 0$, one obtains~\eqref{eq:RSK=Burge} in the case where $X$ has i.i.d.\ exponential entries of rate $\alpha$.
\end{proof}

\begin{proof}[Proof of Theorem~\ref{thm:UpDownLPP}]
Let $\Fro$ be the border strip of $\lambda$.
Using~\eqref{eq:RSK_Burge_Greene}
for $k=1$, we see that $L_{m,n} = \RSK(X)_{m,n}$ and $L^*_{m,n} = \Burge(X)_{m,n}$ for all $(m,n)\in\Fro$.
If the weights are i.i.d.\ geometric or exponential variables, then $\RSK(X)$ and $\Burge(X)$ are equal in distribution by Lemma~\ref{lem:RSK=Burge}.
It follows that the restrictions of the LPP and dual LPP tableaux to the border strip, namely $\RSK(X)|_\Fro=L|_\Fro$ and $\Burge(X)|_\Fro=L^*|_\Fro$, are also equal in distribution.
\end{proof}

\section{From Conjecture~\ref{main-conj} to a combinatorial identity}

\label{sec:comb-iden}

In this section we reformulate Conjecture~\ref{main-conj} by showing its equivalence to Conjecture~\ref{main-conj-reformulated}.
We start by discussing the two families of combinatorial objects and defining the relevant associated quantities appearing in identity~\eqref{eq:comb-iden}.

\subsection{Staircase shape Young tableaux}
\label{subsec:staircaseTableaux}
Let $\staircase_n$ denote the partition $(n-1,n-2,\ldots,1)$ of $N=n(n-1)/2$; as a Young diagram we will refer to $\staircase_n$ as the \firstmention{staircase shape of order~$n$}.
Let $\syt(\staircase_n)$ denote the set of standard Young tableaux of shape $\staircase_n$.
We associate with each $t\in\syt(\staircase_n)$ several parameters, which we denote by $\diag_t$, $\sigma_t$, $\deg_t$, and $f_t$.

First, we define $\diag_t := (t_{n-1,1}, t_{n-2,2}, \ldots, t_{1,n-1})$ to be the vector of corner entries of~$t$ read from bottom-left to top-right. 
Second, we define $\sigma_t\in \sym_{n-1}$ to be the permutation  encoding the ordering of the entries of $\diag_t$, so that $\diag_t(j) < \diag_t(k)$ if and only if $\sigma_t(j) < \sigma_t(k)$ for all $j,k$.
The vector $\diagsorted_t$ will denote the increasing rearrangement of $\diag_t$, in the sense that $\diagsorted_t(k):=\diag_t(\sigma_t^{-1}(k))$ for all~$k$.
For later convenience we adopt the notational convention that  $\diagsorted_t(0)=0$.

Recall that a tableau $t \in \syt(\staircase_n)$ encodes a growing sequence
\begin{equation}
\label{eq:diagseq}
\emptyset = \lambda^{(0)} \nearrow \lambda^{(1)}
\nearrow \lambda^{(2)} \nearrow \ldots \nearrow \lambda^{(N)}=
\staircase_n
\end{equation}
of Young diagrams that starts from the empty diagram, ends at $\staircase_n$, and such that each $\lambda^{(k)}$ is obtained from $\lambda^{(k-1)}$ by adding the box $(i,j)$ for which $t_{i,j}=k$.
We then define the vector $\deg_t = (\deg_t(0),\ldots,\deg_t(N-1))$, where $\deg_t(k)$ is the number of boxes $(i,j)\in \staircase_n \setminus \lambda^{(k)}$ such that $\lambda^{(k)} \cup \{(i,j)\}$ is a Young sub-diagram of $\staircase_n$.
We may interpret $\deg_t(k)$ as the out-degree of $\lambda^{(k)}$ regarded as a vertex of the directed graph $\young(\staircase_n)$ of Young diagrams contained in $\staircase_n$ (a sublattice of the \firstmention{Young graph}, or \firstmention{Young lattice}, $\mathcal{Y}$), with edges corresponding to the box-addition relation $\mu \nearrow \lambda$.
Note that, in a simple random walk on $\young(\staircase_n)$ that starts from the empty diagram $\emptyset$, the probability of  the sequence of diagrams \eqref{eq:diagseq} associated with the tableau $t$ is precisely
$\prod_{k=0}^{N-1} \deg_t(k)^{-1}$.

Finally, we define the \firstmention{generating factor} of~$t$ as the rational function
\begin{equation}
f_t(x_1,\ldots,x_{n-1}) :=
\prod_{k=1}^{n-1} \ \ \prod_{\diagsorted_t(k-1) < j \le \diagsorted_t(k)} \frac{1}{x_k + \deg_t(j)} \, .
\label{eq:f_t}
\end{equation}

\definecolor{colorone}{rgb}{0,0.2,0.8}
\definecolor{colortwo}{rgb}{1,0,0}
\definecolor{colorthree}{rgb}{0,0.4,0}
\definecolor{colorfour}{rgb}{0.5,0,0.5}
\definecolor{colorfive}{rgb}{0,0,0}

\begin{ex}
\label{ex:syt6}
For the tableau $t$ shown in Fig.~\ref{fig:syt-sn6} (left), we have 
\begin{align*}
\diag_t &= 
({{\color{colorone}10}, {\color{colorthree}13}, {\color{colorfive}15}, {\color{colorfour}14}, {\color{colortwo}11}}) \, ,
\quad
\sigma_t = 
({{\color{colorone}1}, {\color{colorthree}3}, {\color{colorfive}5}, {\color{colorfour}4}, {\color{colortwo}2}}) \, ,
\quad
\deg_t =
({{\color{colorone}1}, {\color{colorone}2}, {\color{colorone}2}, {\color{colorone}3}, {\color{colorone}3}, {\color{colorone}3}, {\color{colorone}4}, {\color{colorone}4}, {\color{colorone}4}, {\color{colorone}4}, {\color{colortwo}3}, {\color{colorthree}2}, {\color{colorthree}3}, {\color{colorfour}2}, {\color{colorfive}1}}) \, , \\
f_t &=
{\color{colorone}\frac{1}{(x_{1}+{1})(x_{1}+{2})^2 (x_{1}+{3})^3 (x_{1}+{4})^4}} \cdot
{\color{colortwo}\frac{1}{x_{2}+{3}}} \cdot
{\color{colorthree}\frac{1}{(x_{3}+{2})(x_{3}+{3})}} \cdot
{\color{colorfour}\frac{1}{x_{4}+{2}}} \cdot
{\color{colorfive}\frac{1}{x_{5}+{1}}} \, .
\end{align*}
Here, we have used colors to illustrate how the entries of $\diag_t$ determine a decomposition of $\deg_t$ into blocks, which correspond to different variables $x_k$ in the definition of the generating factor $f_t$.
\end{ex}

\setlength{\belowcaptionskip}{-5pt}
\begin{figure}[t!]
\centering
\centering
{\includegraphics[height=30mm]{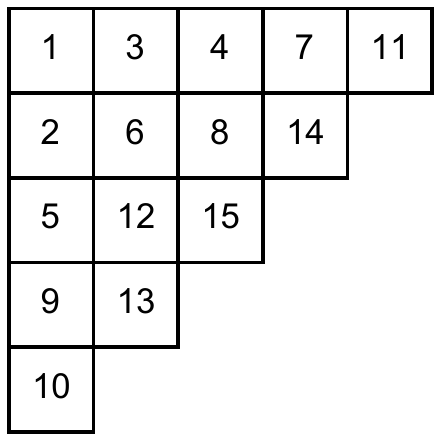}}
\hspace{0.3cm} \raisebox{40pt}{$\xmapsto{\ \eg\ }$} \hspace{0.3cm}
{\includegraphics[height=30mm]{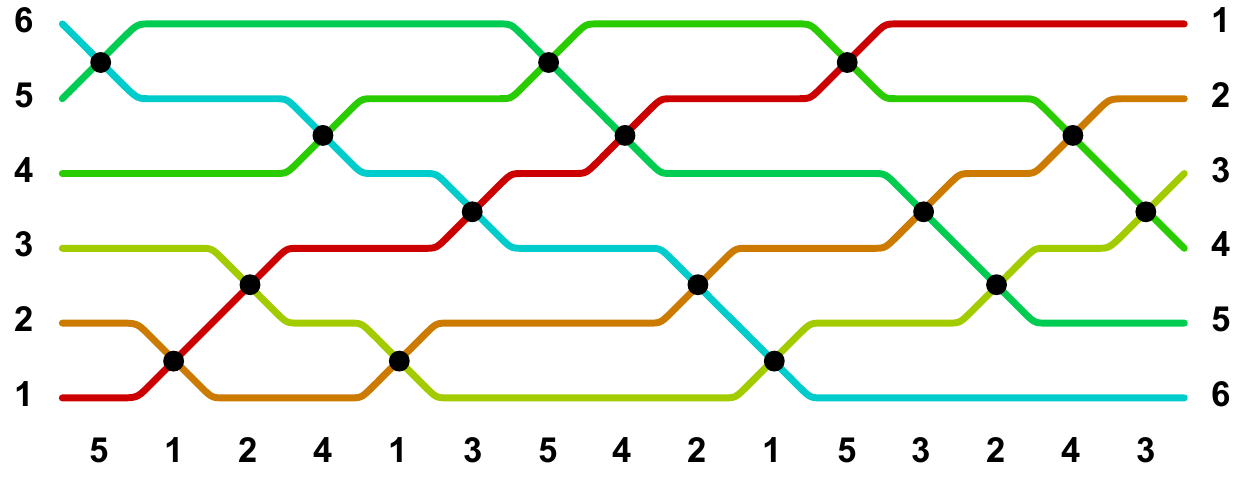}}
\caption{A staircase shape standard Young tableau $t$ of order~6, shown in ``English notation'', and the sorting network $s=\eg(t)$ of order~6 (illustrated graphically as a wiring diagram) associated with the former via the Edelman--Greene correspondence.}
\label{fig:syt-sn6}
\end{figure}

\subsection{Sorting networks}
\label{subsec:sortingNetworks}
Recall that a \firstmention{sorting network of order~$n$} is a synonym for a reduced word decomposition of the reverse permutation $\rev_n = (n,n-1,\ldots,1)$ in terms of the Coxeter generators $\tau_j = (j \ \ j+1)$, $1\le j< n $.
Formally, a sorting network is a sequence of indices $s=(s_1,\ldots,s_{N})$ of length $N=n(n-1)/2$, such that
$ 1\le s_j <n $ for all $j$ and 
$ \rev_n = \tau_{s_{N}} \cdots \tau_{s_2} \tau_{s_1}$.

We denote by $\sort_n$ the set of sorting networks of order~$n$.
The elements of $\sort_n$ can be interpreted as maximal length chains in the weak  Bruhat order or, equivalently, shortest paths in the poset lattice (which is the Cayley graph of $\sym_n$ with the adjacent transpositions as generators) connecting the identity permutation $\id_n$ to the permutation $\rev_n$.
They can be portrayed graphically using \firstmention{wiring diagrams}, as illustrated in Fig.~\ref{fig:syt-sn6}.

Stanley~\cite{stanley84} proved that sorting networks are equinumerous with staircase shape Young tableaux of the same order, i.e.\ $|\sort_n|=|\syt(\staircase_n)|$.
Edelman and Greene~\cite{edelmanGreene87} found an explicit bijection $\eg:\syt(\staircase_n)\to\sort_n$, known as the Edelman--Greene correspondence.

We associate with a sorting network $s \in \sort_n$ parameters $\fin_s$, $\pi_s$, $\deg_s$, and $g_s$ that will play a role analogous to the parameters $\diag_t$, $\sigma_t$, $\deg_t$, and $f_t$ for $t\in \syt(\staircase_n)$.

We define the vector $ \fin_s = (\fin_s(1),\fin_s(2),\ldots,\fin_s(n-1))$ by setting $\fin_s(k) := \max\{ 1\le j \le N\,:\, s_j = k \}$ to be the index of the last swap occurring between positions $k$ and $k+1$.
We define $\pi_s \in \sym_{n-1}$ to be the permutation encoding the ordering of the entries of $\fin_s$, so that $\fin_s(j) < \fin_s(k)$ if and only if $\pi_s(j) < \pi_s(k)$. 
We denote by $\finsorted_s$ the increasing rearrangement of $\fin_s$, and we use the notational convention $\finsorted_s(0)=0$.

We next define $\deg_s = (\deg_s(0),\ldots,\deg_s(N-1))$ to be the vector with coordinates
$\deg_s(k) := |\{ 1\le j\le n-1\colon \nu^{(k)}(j)<\nu^{(k)}(j+1)\}|$,
where $\nu^{(k)} := \tau_{s_k}\cdots \tau_{s_2}\tau_{s_1}$ is the $k$-th permutation in the path encoded by $s$.
Note that $\deg_s(k)$ is the out-degree of $\nu^{(k)}$ in the weak Bruhat order of $S_n$ considered as a directed graph, where edges are directed in the direction of increasing distance from $\id_n$.

Finally, the \firstmention{generating factor} $g_s$ of $s$ is defined, analogously to~\eqref{eq:f_t}, as the rational function
\begin{equation}
g_s(x_1,\ldots,x_{n-1}) = \prod_{k=1}^{n-1} \ \ \prod_{\finsorted_s(k-1) < j \le \finsorted_s(k)} \frac{1}{x_k + \deg_s(j)} \, .
\label{eq:g_s}
\end{equation}

\begin{ex} 
\label{ex:sortingnet6}
The sorting network $s = (5, 1, 2, 4, 1, 3, 5, 4, 2, 1, 5, 3, 2, 4, 3) \in \sort_6$ associated via the Edelman--Greene correspondence with the tableau $t$ from Example~\ref{ex:syt6} is shown in Fig.~\ref{fig:syt-sn6} (right).
Its parameters (shown using color coding as in Example~\ref{ex:syt6}) are
\begin{align*}
\fin_s &
= ({{\color{colorone}10}, {\color{colorthree}13}, {\color{colorfive}15}, {\color{colorfour}14}, {\color{colortwo}11}}),
\quad
\pi_s = ({{\color{colorone}1}, {\color{colorthree}3}, {\color{colorfive}5}, {\color{colorfour}4}, {\color{colortwo}2}}),
\quad
\deg_s
 = ({{\color{colorone}5}, {\color{colorone}4}, {\color{colorone}3}, {\color{colorone}3}, {\color{colorone}3}, {\color{colorone}2}, {\color{colorone}3}, {\color{colorone}2}, {\color{colorone}2}, {\color{colorone}3}, {\color{colortwo}2}, {\color{colorthree}1}, {\color{colorthree}2}, {\color{colorfour}1}, {\color{colorfive}1}}),
\\
g_s &= 
{\color{colorone}\frac{1}{(x_{1}+{5})(x_{1}+{4})(x_{1}+{3})^5(x_{1}+{2})^3}} \cdot
{\color{colortwo}\frac{1}{x_{2}+{2}}} \cdot
{\color{colorthree}\frac{1}{(x_{3}+{1})(x_{3}+{2})}} \cdot
{\color{colorfour}\frac{1}{x_{4}+{1}}} \cdot
{\color{colorfive}\frac{1}{x_{5}+{1}}} \, .
\end{align*}
\end{ex}

The following result is easy to guess from Examples~\ref{ex:syt6} and~\ref{ex:sortingnet6}.

\begin{lem}
\label{lem:edelman-greene-params}
If $t\in \syt_n$ and $s=\eg(t) \in \sort_n$ then we have that
\begin{equation}
\label{eq:edelman-greene-params}
\fin_s = \diag_t \, , \qquad\qquad
\pi_s = \sigma_t \, .
\end{equation}
\end{lem}

\begin{proof}
The second relation follows trivially from the first.
The first relation is an easy consequence of the definition of the Edelman--Greene correspondence, and specifically of the way the map $\eg$ can be visualized as ``emptying'' the tableau $t$ by repeatedly applying the Sch\"utzenberger operator. In the notation of  \cite[\S~4]{angelHolroydRomikVirag07}, we have:
\begin{align*}
\fin_s(k) 
&= \max \{ 1\le m\le N \,:\, j_\textrm{max}(\Phi^{N-m}(t)) = k \} \\
&= N - \min \{ 0\le r<N\,:\, j_\textrm{max}(\Phi^r(t)) = k \} \\
&= N- (N- t_{n-k,k}) = t_{n-k,k} = \diag_t(k) \, .
\qedhere
\end{align*}
\end{proof}

\vspace{-15pt}

\subsection{The combinatorial identity}
\label{subsec:combinIdentity}

Let $\freevecspace_{n-1}$ denote the free vector space generated by the elements of $\sym_{n-1}$ over the field of rational functions $\C_\mathbf{x}^{n-1}:=\C(x_1,\ldots,x_{n-1})$.
Define the following generating functions as elements of $\freevecspace_{n-1}$:
\begin{align}
F_n(x_1,\ldots,x_{n-1}) &:=\!\!\! \sum_{t\in \syt(\staircase_n)} f_t(x_1,\ldots,x_{n-1}) \sigma_t \, ,
\label{eq:genfun-F}
\\
G_n(x_1,\ldots,x_{n-1}) &:= \sum_{s\in \sort_n} 
g_s(x_1,\ldots,x_{n-1}) \pi_s \, .
\label{eq:genfun-G}
\end{align}
Conjecture~\ref{main-conj-reformulated} is the identity $F_n(x_1,\ldots,x_{n-1}) = G_n(x_1,\ldots,x_{n-1})$ (an equality of vectors with $(n-1)!$ components).
Note that in general it is \emph{not} true that $f_t = g_s$ if $s=\eg(t)$, as Examples~\ref{ex:syt6} and~\ref{ex:sortingnet6} clearly show. Thus, the Edelman--Greene correspondence does not seem to imply the conjecture in an obvious way. However, using~\eqref{eq:edelman-greene-params} we see that the  correspondence does imply the limiting case $\lim\limits_{x \to\infty} x^N (F_n(x,\ldots,x)-G_n(x,\dots,x))=0$.

The calculation of $F_n(x_1,\ldots,x_{n-1})$ and $G_n(x_1,\ldots,x_{n-1})$ involves a summation over $|\syt(\staircase_n)|=|\sort_n|=N!/(1^{n-1}\cdot3^{n-2}\cdots(2n-3)^1)$ elements. For $n\leq 6$ this calculation is feasible by using symbolic algebra software.
We wrote code in Mathematica---downloadable as a companion package \cite{orientedswaps-mathematica} to this extended abstract---to perform this calculation and check that the two functions are equal, thus 
proving Theorem~\ref{thm:main-thm2}.

\begin{figure}[t!]
\centering
{\includegraphics[width=1\columnwidth]{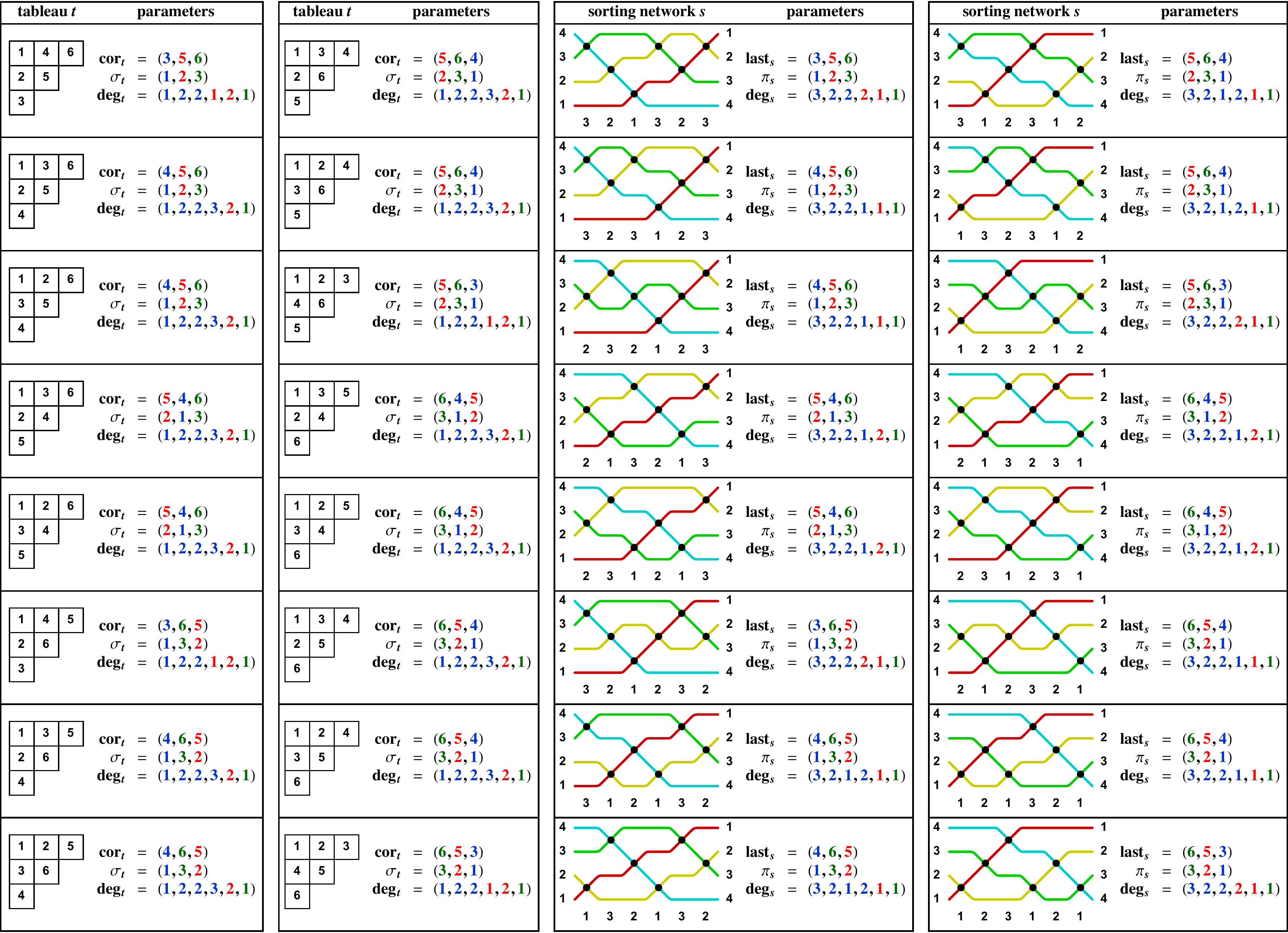}}
\caption{The $16$ staircase shape standard Young tableaux and sorting networks of order $4$ (ordered so that entries in the same relative positions in the two tables correspond to each others via the Edelman--Greene correspondence). As in Examples~\ref{ex:syt6}--\ref{ex:sortingnet6}, the coloring of the parameter entries emphasizes how different entries of $\deg_t$ and $\deg_s$ correspond to different factors in the definition of the generating factors $f_t$ and $g_s$.
}
\label{fig:syt-sn4}
\end{figure}

\begin{ex}
For $n=4$, the generating functions can be computed by hand using the tables shown in Fig.~\ref{fig:syt-sn4} above.
For example, 
the component of the two generating functions associated with the identity permutation $\operatorname{id}=(1,2,3)$ is
\[
\big(F_4(x_1,x_2,x_3)\big)_{\operatorname{id}} \!=\! \big(G_4(x_1,x_2,x_3)\big)_{\operatorname{id}} \!=\! 
\frac{x_1 + 2x_2 + 5}{(x_1+1) (x_1+2)^2 (x_1+3) (x_2+1) (x_2+2) (x_3+1)} \, .
\]
\end{ex}

\subsection{Equivalence of the combinatorial and probabilistic conjectures}
\label{subsec:equivalenceConjectures}

We now sketch the proof of the equivalence between Conjectures~\ref{main-conj} and~\ref{main-conj-reformulated}.
A key insight is that we can write explicit formulas for the density functions of $\bm{V}_n$ and $\bm{U}_n$ by interpreting both the randomly growing Young diagram model and the oriented swap process as continuous-time random walks on the directed graph $\young(\staircase_n)$ of Young sub-diagrams of $\staircase_n$, and the Cayley graph of $S_n$ with Coxeter generators, respectively. 
Specifically, we write the density function of $\bm{V}_n$ (resp.\ $\bm{U}_n$) as a weighted average of the conditional densities conditioned on the path that the process takes to get from the initial state $\emptyset$ (resp.\  $\id_n$) to the final state $\staircase_n$ (resp.\  $\rev_n$), that is,
\begin{align*}
p_{\bm{V}_n}(v_1,\ldots,v_{n-1})&= \sum_{t \in \syt(\staircase_n)} \P(T=t) \, p_{\bm{V}_n|T=t}(v_1,\ldots,v_{n-1}) \, ,
\\
 p_{\bm{U}_n}(u_1,\ldots,u_{n-1}) &= \sum_{s \in \sort_n} \P(S=s)  \, p_{\bm{U}_n|S=s}(u_1,\ldots,u_{n-1}) \, .
\end{align*}
Here, $t$ (resp.\ $s$) can be viewed as a realization of a (discrete-time) simple random walk $T$ (resp.\  $S$) on the directed graph $\young(\staircase_n)$ (resp. on the Cayley graph of $S_n$); therefore, $\P(\nobreak{T=t}) = \prod_{j} \deg_t(j)^{-1}$ and  $\P(S=s)=\prod_{j} \deg_s(j)^{-1}$.
Conditioned on this combinatorial path, the continuous time processes are a time-reparametrization of the discrete-time random walks. In fact, it is not hard to show that the conditional densities $p_{\bm{V}_n|T=t}$ and $p_{\bm{U}_n|S=s}$  are completely determined by the vectors  $\diagsorted_t$ and $\finsorted_s$ and their relative ordering $\sigma_t$ and $\pi_s$ in the simple random walks, and the sequences of out-degrees $\deg_t$ and $\deg_s$ along the paths (which correspond to the exponential clock rates to leave each vertex in the graph where the random walk is taking place).
The formulas for  the density functions of $\bm{U}_n$ and of $\bm{V}_n$ thus take the form
\newcommand{\convolutionop}{\mathop{\mathlarger{\mathlarger{\mathlarger{*}}}}}
\begin{align}
p_{\bm{V}_n}(v_1,\ldots,v_{n-1})
&= 
\!\!\!\sum_{t \in \syt(\staircase_n)} 
\frac{1}{\displaystyle\prod_{j=0}^{N-1} \deg_t(j) }\displaystyle
\prod_{k=1}^{n-1}\left(\convolutionop_{j=\diagsorted_t(k-1)+1}^{\diagsorted_t(k)} 
E_{\deg_t(j)} \right) \left(v_{\sigma_t^{-1}(k)}-v_{\sigma_t^{-1}(k-1)}\right) ,\nonumber
\\
p_{\bm{U}_n}(u_1,\ldots,u_{n-1})
&= 
\sum_{s \in \sort_n} \frac{1}{\displaystyle\prod_{j=0}^{N-1} \deg_s(j)}
\prod_{k=1}^{n-1}  \left(\convolutionop_{j=\finsorted_s(k-1)+1}^{\finsorted_s(k)} 
E_{\deg_s(j)}\right) \left(u_{\pi_s^{-1}(k)}-u_{\pi_s^{-1}(k-1)}\right) ,\nonumber
\label{eq:joint-density-un_vn}
\end{align}
where the notation $\convolutionop\limits_{j=1}^m f_j$ is a shorthand for the convolution $f_1 * \ldots * f_m$ of one-dimensional densities; and $E_\rho(u)= \rho e^{-\rho u} \1_{[0,\infty)}(u)$ is the exponential density with parameter $\rho>0$. 

Now, Conjecture~\ref{main-conj} can be viewed as claiming the equality $p_{\bm{U}_n} = p_{\bm{V}_n}$ of the joint density functions of $\bm{U}_n$ and $\bm{V}_n$, or equivalently the equality $\hat{p_{\bm{U}_n}} = \hat{p_{\bm{V}_n}}$ of the corresponding Fourier transforms.
In turn, the latter can be manipulated and recast as the combinatorial identity $F_n=G_n$ of Conjecture~\ref{main-conj-reformulated}, using the fact that the Fourier transform of the exponential density is $\hat{E_\rho}(x)=\rho/(\rho+ix)$.

\printbibliography

\end{document}